\documentclass[reqno]{amsart}

\usepackage{amsmath,amssymb,geometry,lipsum,multicol}

\newtheorem{theorem}{Theorem}[section]
\newtheorem{lemma}[theorem]{Lemma}

\newtheorem{proposition}[theorem]{Proposition}

\theoremstyle{definition}
\newtheorem{remark}[theorem]{Remark}

\numberwithin{equation}{section}

\usepackage{tikz}

\newcommand{\BC}{{\mathbb C}}
\newcommand{\BF}{{\mathbb F}}
\newcommand{\BH}{{\mathbb H}}

\newcommand{\BL}{{\mathbb L}}

\newcommand{\BR}{{\mathbb R}}

\newcommand{\cC}{{\mathcal C}}
\newcommand{\cE}{{\mathcal E}}
\newcommand{\cH}{{\mathcal H}}

\newcommand{\cL}{{\mathcal L}}

\newcommand{\cP}{{\mathcal P}}

\newcommand{\cU}{{\mathcal U}}\newcommand{\cV}{{\mathcal V}}
\newcommand{\cW}{{\mathcal W}}
\newcommand{\cZ}{{\mathcal Z}}

\newcommand{\bC}{{\mathbf C}}

\newcommand{\bR}{{\mathbf R}}

\newcommand{\fA}{{\mathfrak A}}

\newcommand{\fY}{{\mathfrak Y}}

\newcommand{\whatC}{\widehat{C}}

\newcommand{\whatS}{\widehat{S}}

\newcommand{\la}{\lambda}

\newcommand{\rank}{\textup{rank\,}}
\newcommand{\ran}{\textup{ran\,}}

\newcommand{\kr}{\textup{Ker\,}}

\newcommand{\mat}[1]{\begin{bmatrix} #1 \end{bmatrix}}

\newcommand{\ov}[1]{{\overline{#1}}}

\newcommand{\tu}[1]{\textup{#1}}
\newcommand{\wtil}[1]{{\widetilde{#1}}}

\newcommand{\ands}{\quad\mbox{and}\quad}

\newcommand{\PRO}{\mathcal{PRO}}

\newcommand{\cic}{\textsf{cic}}

\newcommand{\vect}{\operatorname{vec}}
\newcommand{\imag}{\textup{i}}

\begin{document}
\title{
Convex invertible cones and Nevanlinna-Pick interpolation:\\ The suboptimal case
}

\thanks{This work is based on research supported in part by the National Research Foundation of South Africa (Grant Numbers 118513 and 127364).}

\author[S. ter Horst]{S. ter Horst}
\address{S. ter Horst, Department of Mathematics, Research Focus Area:\ Pure and Applied Analytics, North-West University, Potchefstroom, 2531 South Africa and DSI-NRF Centre of Excellence in Mathematical and Statistical Sciences (CoE-MaSS)}
\email{Sanne.TerHorst@nwu.ac.za}

\author[A. van der Merwe]{A. van der Merwe}
\address{A. van der Merwe, Faculty of Engineering and the Built Environment, Academic Development Unit, University of the Witwatersrand, Johannesburg, 2000 South Africa and DSI-NRF Centre of Excellence in Mathematical and Statistical Sciences (CoE-MaSS)}
\email{alma.vandermerwe@wits.ac.za}

\begin{abstract}                
Nevanlinna-Pick interpolation developed from a topic in classical complex analysis to a useful tool for solving various problems in control theory and electrical engineering. Over the years many extensions of the original problem were considered, including extensions to different function spaces, nonstationary problems, several variable settings and interpolation with matrix and operator points. Here we discuss a variation on Nevanlinna-Pick interpolation for positive real odd functions evaluated in real matrix points. This problem was studied by Cohen and Lewkowicz using convex invertible cones and the Lyapunov order, but was never fully resolved. In this paper we present a solution to this problem in a special case that we refer to as `suboptimal' based on connections with the classical case. The solution requires a representation of linear matrix maps going back to R.D. Hill and an analysis of when positive linear matrix maps are completely positive, on which we reported in earlier work and which we will briefly review here.
\end{abstract}



\maketitle
\begin{multicols}{2}
\section{Introduction}

Nevanlinna-Pick interpolation refers to a metric constrained interpolation problem in complex analysis that was solved by \cite{N29} and \cite{P16} in the early 20th century. The fact that it remains an active field of study is largely due to its importance in various applications in control theory and electrical engineering; cf., \cite{DGK81} for an account of some early applications and \cite{MG18,BOS17,WYO14,BtH10b} for a few more recent developments. As a specific classical application we mention the early work of \cite{F24,C26,B31} in which the impedances of lumped electrical one-port circuits generated by inductances and capacitors ware shown to correspond to the class of positive real odd functions; in electrical engineering ``odd'' is usually indicated by ``lossless,'' however, we shall follow here the terminology used by \cite{CL07}. This connection between electrical networks and complex functions motivated \cite{YS67} to investigate Nevanlinna-Pick interpolation in this class of functions. In particular, they developed more explicit formulas for the solutions that enables one to construct such electrical circuits with certain constrains imposed by the interpolation conditions.

Later applications, for instance in multi-port circuits, multidimensional systems, nonlinear and non-stationary systems, required more advanced variations on Nevanlinna-Pick interpolation, e.g., for matrix or operator functions, other function classes and more intricate domains, multivariable extensions, nonlinear and non-stationary versions, evaluation in matrix and operator points, and led to a renewed interest for the problem in the past few decades.\\

\noindent{\em Interpolation in matrix points.}
The classical Nevanlinna-Pick interpolation problem asks for a function from a given function class that takes prescribed values at certain points in a given domain, together with a metric constraint. Evaluating a function in a matrix point, say in a Jordan block, also requires evaluation of the derivatives of the function at the eigenvalue of the Jordan block up to the size of the block minus one. Therefore, interpolation conditions in matrix points enables one to combine classical interpolation conditions with conditions where the values of the derivatives up to a certain order are prescribed, which can all be encoded in the Jordan structure of a single pair of matrices. It is this kind of Nevanlinna-Pick interpolation that was studied by \cite{CL09}, as we will explain in the next section. \\

{\em Summary of the paper.} The paper is organized as follows. In Section \ref{S:CLapproach} we formulate the interpolation problem considered by Cohen and Lewkowicz as well as various aspects of their approach, like convex invertible cones and the Lyapunov order. Section \ref{S:Main} contains our main result as well as various definitions required to state the result. The result consists of four equivalent statements. One implication comes from the work of Cohen and Lewkowicz and will be explained in Section \ref{S:CLapproach}. Two other implications have appeared in other work of the authors which will be shortly reviewed in Sections \ref{S:HillRep} and \ref{S:P=CP}, in which also some important results and observations for the proof of the remaining implication will be made. This last implication will be proved in Section \ref{S:Proof}.

\section{The Cohen-Lewkowicz cic approach to Nevanlinna-Pick interpolation}\label{S:CLapproach}

N. Cohen and I. Lewkowicz (in \cite{CL07,CL09}) considered a variation on Nevanlinna-Pick interpolation for real rational functions within the class of {\em positive real odd functions}, denoted $\PRO$; recall that functions $f$ in $\PRO$ are real rational functions that take nonnegative values on the positive real line and satisfy $f(-\ov{z})^*=-f(z)$ for $z\in\BC$ not poles of $f$ (so that they are odd functions on the real line). For two matrices $A,B\in\BR^{n \times n}$ the question is when there exists a $f\in\PRO$ so that $f(A)=B$.\\

{\em Convex invertible cones.}
The underlying concept in the approach to Nevanlinna-Pick interpolation by Cohen and Lewkowicz is that of a convex invertible cone, abbreviated here to \cic. In a unital algebra $\fA$, a \cic\ is a convex cone that is closed under inversion.  For a set $X$ in $\fA$, the \cic\  generated by $X$ is the smallest \cic\ that contains $X$ and is denoted as $\cC(X)$; we write $\cC(a)$ instead of $\cC(\{a\})$ in case the set is a singleton.

For instance, $\PRO$ forms a \cic\ within the algebra of real rational functions. In fact, it is a \cic\ that is singly generated by the function $f_{\infty}(z)=z$; so $\PRO=\cC(f_\infty)$. This fact immediately leads to an important insight for the Nevanlinna-Pick interpolation problem:\ For $A\in\BR^{n\times n}$ with no poles on the imaginary axis $\imag \BR$ (because $\PRO$-functions have poles and zeros only on $\imag \BR$), the set $\PRO(A):=\{f(A)\colon f\in\PRO\}$ is a \cic\ within $\BR^{n\times n}$ and we have
\[
\PRO(A)=\cC(f_\infty(A))=\cC(A).
\]
Therefore, we have
\[
B\in\PRO(A) \ \Longleftrightarrow \ B\in\cC(A).
\]
So to solve the interpolation problem, that is whether $B\in\PRO(A)$, we may thus also ask whether $B\in\cC(A)$. However, the \cic\ $\cC(A)$ is not easily described when $A$ is nonsingular. Note that the bicommutant $\{A\}''$ of $A$ is also a \cic\ that contains $A$, so that $\cC(A)\subset \{A\}''$, and hence $B$ must be in $\{A\}''$ whenever $B\in\PRO(A)$. See \cite{CL07} for more details on the \cic-structure of $\PRO$ and \cite{tHN21} for the case of matrix $\PRO$ functions.

A further constraint imposed by the interpolation problem is caused by the function class $\PRO$. For $f\in\PRO$ and $z\in\BC$, if we know $f(z)$, then the fact that $f$ is real rational and odd directly specifies $f$ in $\ov{z}$, $-z$ and $-\ov{z}$. Hence, if we specify an interpolation condition at one of these four points, we should not specify a possibly conflicting condition at one of the other three. Therefore, we assume $A$ to be {\em Lyapunov regular}, that is, for the eigenvalues $\la_1,\ldots,\la_n$ of $A$ we demand
\[
\la_i+\ov{\la}_j \ne 0,\qquad i,j=1, \ldots,n.
\]
Note that this implies $A$ cannot have eigenvalues on the imaginary axis $\imag \BR$, in particular, $A$ is nonsingular.

While the Nevanlinna-Pick problem is stated specifically for the case of real matrices, much of what we need to solve the problem works over the complex numbers as well, and in the remainder we will only make the distinction when necessary. So, throughout this paper we use $\BF$ to indicate $\BC$ or $\BR$. To avoid confusion we shall use notation and terminology as if $\BF=\BC$. \\

{\em The Lyapunov order.}
We write $\cH_n$ for the set of $n \times n$ Hermitian matrices, $\cP_n$ for the set of all positive definite matrices in $\cH_n$ and $\ov{\cP}_n$ for the set of all positive semidefinite matrices in $\cH_n$. Given $A\in\BF^{n\times n}$, we call $H\in\cH_n$ a solution to the Lyapunov inequality of $A$ (Lyapunov solution of $A$ for short) if
\begin{align*}
&HA + A^*H \in \cP_n\ \ (\text{strict}) \ \  \text{or} \\  &HA + A^*H \in \ov{\cP}_n\ \ (\text{non-strict}).
\end{align*}
The Lyapunov solution sets of $A$ are thus given by as
\begin{align*}
\cH(A)&=\{H \in \cH_n: HA + A^*H \in \cP_n\},\\
\ov{\cH}(A)&=\{H \in \cH_n: HA + A^*H \in \ov{\cP}_n\}.
\end{align*}
It is easy to see that $\cH(A)$ and $\ov{\cH}(A)$ are also matrix \cic s.
\cite{CL09} introduced the {\em Lyapunov order} (for $\BF=\BR$), namely, for matrices $A,B \in \BF^{n \times n}$, it is said that $B$ {\em Lyapunov dominates} $A$, denoted $A \le_{\cL} B$, if
\begin{equation}\label{LyapOrder}
\ov{\cH}(A)\subset \ov{\cH}(B).
\end{equation}
Hence the Lyapunov order $A \le_\cL B$ means that all non-strict Lyapunov solutions of $A$ are also non-strict Lyapunov solutions of $B$. If $A$ and $B$ are both Lyapunov regular, \eqref{LyapOrder} is equivalent to $\cH(A)\subset \cH(B)$. Furthermore, the set
\[
\mathcal{C}_\cL(A)=\{B \in \BF^{n \times n}: A \le_\cL B\}
\]
is also a \cic\ and it clearly contains $A$. Since the intersection of \cic s is again a \cic, we have the inclusion
\[
\mathcal{C}(A) \subseteq \mathcal{C}_\cL(A)\bigcap \{A\}''_{\BF}.
\]
The subscript $\BF$ is added to the bicommutant of $A$ since if $A$ is real, one can consider the commutant both in $\BC^{n \times n}$ and $\BR^{n \times n}$. In \cite{CL09} Cohen and Lewkowicz conjectured that these two sets coincide (for $\BF=\BR$), at least when $A$ is Lyapunov regular. To test whether $A \le_\cL B$ when $A$ is Lyapunov regular (and $\BF=\BR$), Cohen and Lewkowicz formulated the weak Pick test as verifying that all extremal elements of $\ov{\cH}(A)$ are also in $\ov{\cH}(B)$, while if in addition $B\in\{A\}''_{\BF}$, they formulate the strong Pick test in which one only has to verify inclusion of a single extremal element of $\ov{\cH}(A)$ \cite[p.\ 1852]{CL09}.

In \cite{tHvdM22b}, for $\BF=\BR,\BC$, we showed that when $A$ is Lyapunov regular and  $B\in\{A\}''_{\BF}$, there exists a Hermitian matrix $\BH_{A,B}$, that can be explicitly computed, so that $A \le_\cL B$ corresponds to $\BH_{A,B}$ being positive semidefinite. The matrix $\BH_{A,B}$, called the Hill-Pick matrix of the pair $(A,B)$, is of size $m_\tu{max} \times m_\tu{max}$ with $m_\tu{max}=\dim \{A\}_\BF''$, where $\{A\}_\BF''$ is viewed as a linear subspace of $\BF^{n \times n}$. More concretely, $m_\tu{max}$ corresponds to the sum of the sizes of the Jordan blocks in the Jordan decomposition of $A$ where in case there are multiple Jordan blocks with the same eigenvalue, only the largest one is taken into account, and the others are discarded. See Section \ref{S:HillRep} below for more details.

\section{Main result}\label{S:Main}

\cite{vdMthesis} showed that the conjecture of Cohen and Lewkowicz is true (for $\BF=\BR$) in case $B$ does not simultaneously  has $1$ and $-1$ as an eigenvalue. In the present paper we provide a proof under a different condition, that we call ``the suboptimal case,'' which corresponds to the case where the Hill-Pick matrix is invertible:\ $\rank \BH_{A,B}=m_\tu{max}=\dim \{A\}_\BF''$. The proof in this case  is more direct and also gives a procedure to construct a solution.

To state our main result we need to reformulate the Lyapunov order in terms of linear matrix maps. Given a square matrix $Y \in \BF^{n \times n}$, we define the {\em Lyapunov operator} associated with $Y$ by
\begin{equation*}\label{LyapOp}
\cL_Y: \BF^{n \times n} \to \BF^{n \times n}, \quad \cL_Y(X)=XY + Y^*X,
\end{equation*}
with $X \in \BF^{n \times n}$. The Lyapunov operator associated with $Y$ is a linear matrix map which is bijective precisely when $Y$ is Lyapunov regular \cite[Corollary 4.4.7]{HJ91}. Note that $\ov{\cH}(Y)=\cL_Y^{-1}(\ov{\cP}_n)$.

When $A\in\BF^{n \times n}$ is Lyapunov regular, it then follows that $B\in \BF^{n \times n}$ Lyapunov dominates $A$ whenever the linear map
\begin{equation}\label{cL_AB}
\cL_{A,B}=\cL_B \circ \cL_A^{-1}: \BF^{n \times n}\to \BF^{n \times n}
\end{equation}
maps $\ov{\cP}_n$ into $\ov{\cP}_n$, that is, when $\cL_{A,B}$ is a positive linear matrix map. Verifying whether a linear matrix map is positive can be very challenging. Recall that $\cL_{A,B}$ is completely positive if for each positive integer $k$ the map $\cL_{A,B} \otimes I_k:\BF^{nk \times nk} \to \BF^{nk \times nk}$, with $\otimes$ the Kronecker tensor product and $I_k$ the $k \times k$ identity matrix, is positive and the complete positivity of $\cL_{A,B}$ can easily be verified by testing if the associated Choi matrix $\BL_{A,B}$ of $\cL_{A,B}$ is positive semidefinite; see Section \ref{S:HillRep} for the definition of the Choi matrix. It turns out that in case $A$ is Lyapunov regular and $B\in \{A\}_{\BF}''$, positivity of $\cL_{A,B}$ corresponds to complete positivity (see \cite{tHvdM22b}).  This turns out to be an important step in the proof of our main theorem.

\begin{theorem} \label{T:Main}
Let $A\in\BR^{n \times n}$ be Lyapunov regular and $B \in \{A\}''_{\BR}$. Let $\BH_{A,B}$ be the Hill-Pick matrix of the pair $(A,B)$ and assume that $\rank \BH_{A,B}=\dim\{A\}_\BR''$. Then the following statements are equivalent. \begin{itemize}
     \item[(i)] $f(A)=B$ for a function $f \in \mathcal{PRO},$
     \item[(ii)] $\cL_{A,B}$ is a positive linear map (i.e., $A \le_\cL B$),
     \item[(iii)] $\cL_{A,B}$ is a completely positive linear map,
     \item[(iv)] $\BH_{A,B}$ is a positive definite matrix.
 \end{itemize}
\end{theorem}


The implication (i) $\Rightarrow$ (ii) is due to \cite{CL09} and was explained above, (ii) $\Rightarrow$ (iii) comes from \cite{tHvdM22b} and (iii) $\Rightarrow$ (iv) goes back to work of Oxenrider and Hill (\cite{OH85}, cf., \cite{tHvdM22a}), all without the constraint that $\BH_{A,B}$ be invertible. So it remains here to prove the implication (iv) $\Rightarrow$ (i), which we will do in Section \ref{S:Proof}. Nonetheless, in the next few sections we will review some of the work that is used in proving the other implications as well. This requires in particular new insights in linear matrix maps, and some of the observations are needed for the proof of the remaining implication.

\section{$*$-Linear matrix maps and Hill representations}\label{S:HillRep}

A linear matrix map \begin{equation} \label{cL} \cL: \BF^{q \times q} \to \BF^{n \times n}\end{equation}
is called $*$-linear if $\cL(V^*)=\cL(V)^*$ for all $V \in \BF^{q \times q}$. Given the linear matrix map $\cL$ in \eqref{cL} we define  the {\em Choi matrix} $\BL$ given by
\begin{align*}
\BL=\left[\BL_{ij}\right] \in \BF^{nq \times nq},\quad \BL_{ij}=\cL\left(\cE_{ij}^{(q)}\right) \in\BF^{n \times n},
\end{align*}
where $\cE_{ij}^{(q)}$ is the standard basis element in $\BF^{q \times q}$ with a 1 on position $(i,j)$ and zeros elsewhere, and we define the {\em matricization} of $\cL$ as the matrix $L\in \BF^{n^2 \times q^2}$ determined by the linear map
\begin{align}\label{1Matricization}\begin{aligned}
&L:\BF^{q^2} \to \BF^{n^2},\quad \text{where} \\  &L\,\left(\vect_{q\times q}(V)\right)= \vect_{n \times n}\left(\cL (V)\right),
\end{aligned}\end{align}
with $\vect_{r\times s} :\BF^{r \times s} \to \BF^{rs}$ the vectorization operator.  The matrices $\BL$ and $L$ are related through the matrix reordering of \cite{PH81,OH85}, which is discussed in \cite{tHvdM22a}, where more details on what is discussed in this section can be found. The classical result of \cite{C75} states that $\cL$ is completely positive if and only if $\BL$ is positive semidefinite. Also, the map $\cL$ is $*$-linear if and only if $\BL$ is in $\cH_{nq}$.\\

{\em Hill-representations.} In the paper \cite{H73} $*$-linear matrix maps $\cL$ are studied via representations of the form
\begin{equation}\label{HillRep}
\cL(V)=\sum_{k,l=1}^m \BH_{kl}\, C_k V C_l^*,\quad V\in\BF^{q \times q},
\end{equation}
for matrices $C_1,\ldots,C_m \in\BF^{n \times q}$. The matrix $\BH=\left[\BH_{kl}\right]_{k,l=1}^m\in\BF^{m\times m}$ is called the Hill matrix associated with the representation \eqref{HillRep}. Moreover, we call the Hill representation \eqref{HillRep} of $\cL$ minimal if $m$ is the smallest number of matrices $C_k$ that occurs in Hill representations for $\cL$, in which case $C_1,\ldots,C_m$ are linearly independent and $\BH$ is invertible, and it turns out that this smallest number equals the rank of the Choi matrix of $\cL$. In that case $\cL$ is $*$-linear precisely when $\BH$ is in $\cH_m$ and $\cL$ is completely positive if and only if $\BH$ is positive definite. Note that in the latter case, the Hill representation corresponds to a Choi-Kraus representation when $\BH=I_n$, but for us it is not convenient to restrict to this case.

If $\cL$ is given by a minimal Hill representation \eqref{HillRep}, then $L$ and $\BL$ can be represented as
\begin{equation}\label{L-BL Hill reps}
L=\sum_{k,l=1}^m \BH_{kl} \ov{C}_k \otimes C_l \ands \BL=\whatC^* \BH^T \whatC
\end{equation}
with
\begin{equation}\label{whatC}
\widehat{C}^*:=\begin{bmatrix} \vect_{n \times q}\left(C_1\right) & \hdots & \vect_{n \times q}\left(C_m\right)  \end{bmatrix},
\end{equation}
where $\widehat{C}^* \in \BF^{nq \times m}$ has full column rank. The bar in the formula for $L$ means taking complex conjugates in all entries.

Write
\begin{equation}\label{Ldec}
L=[L_{ij}]_{i=1,\ldots,n}^{j=1,\ldots,q} \mbox{ with } L_{ij}\in\BF^{n \times q}.
\end{equation}
It follows from \eqref{L-BL Hill reps} that $C_1,\ldots,C_m$ are all included in
\begin{equation}\label{cW}
\cW:=\tu{span}\{L_{ij}\}.
\end{equation}
In fact, when the Hill representation \eqref{HillRep} is minimal we have
\[
\tu{span}\{C_1,\ldots,C_m\}=\cW
\]
and, moreover, for any $C_1,\ldots,C_m\in \BF^{n \times q}$ with this property there exists a $\BH\in\cH_m$ so that $\cL$ is as in \eqref{HillRep}. It further follows from \eqref{L-BL Hill reps} that
\begin{equation}\label{cLincl}
L\in \ov{\cW}\otimes \cW.
\end{equation}

{\em The matrix map $\cL_{A,B}$.} For $A\in\BF^{n\times n}$ and $B\in\{A\}_{\BF}''$, it is easy to see that $\cL_A$ and $\cL_B$ are $*$-linear, so that also $\cL_{A,B}$ is $*$-linear. It is also easy to see that the matricizations $L_A$ of $\cL_A$ and $L_B$ of $\cL_B$ are given by
\[
L_A=A^T\otimes I_n + I_n \otimes A^*,\quad L_B=B^T\otimes I_n + I_n \otimes B^*.
\]
Since $B\in\{A\}_{\BF}''$, both $L_A$ and $L_B$ are in the matrix algebra $\ov{\{A^*\}_{\BF}''}\otimes \{A^*\}_{\BF}''$, and consequently, so is the matricization $L_{A,B}=L_B L_A^{-1}$ of $\cL_{A,B}$.

\begin{lemma}\label{L:LABstruc}
For $A\in\BF^{n\times n}$ and $B\in\{A\}_{\BF}''$, the matricization $L_{A,B}$ of $\cL_{A,B}$ is in $\ov{\{A^*\}_{\BF}''}\otimes \{A^*\}_{\BF}''$.
\end{lemma}

In case $\cL_{A,B}$ is given by a minimal Hill representation
\begin{equation}\label{cLAB-Hill1}
\cL_{A,B}(V)=\sum_{k,l=1}^m  \BH_{A,B}^{kl} C_kVC_l^*,
\end{equation}
with $V \in \BF^{n \times n}$, it follows from the above that $C_1,\ldots,C_m$ are all $\{A^*\}_{\BF}''$, although the subspace $\cW$ in \eqref{cW} may be strictly included in $\{A^*\}_{\BF}''$.

In this setting we indicate the Hill matrix associated with a minimal Hill representation of $\cL_{A,B}$ by $\wtil{\BH}_{A,B}$. Note first of all that minimal Hill representations of $\cL_{A,B}$ are not unique, however, they are unique up to an invertible matrix in $\BF^{m\times m}$, and, in particular, the Hill matrices are all congruent. Therefore, the choice of the minimal Hill representation is not relevant in the proof of Theorem \ref{T:Main}, but in \cite{tHvdM22b} a canonical choice is described. Secondly, the Hill matrix $\wtil{\BH}_{A,B}$ is of size $m\times m$ with $m=\rank \BL_{A,B}$, which may be strictly smaller than $m_\tu{max}$. In case $m<m_\tu{max}$, to obtain the Hill-Pick matrix $\BH_{A,B}$ one works with non-minimal Hill representations and one obtains $\BH_{A,B}$ as an extension of $\wtil{\BH}_{A,B}$ of the same rank $m$. However, our suboptimality condition $\rank H_{A,B}=\dim \{A\}''_{\BF}$ corresponds precisely to $m=m_\tu{max}$ so that in that case $\wtil{\BH}_{A,B}=\BH_{A,B}$ and it suffices to work with minimal Hill representations; for the general case we refer to \cite{tHvdM22b}.

\section{When are positive $*$-linear matrix maps completely positive}\label{S:P=CP}

To see the implication (ii) $\Rightarrow$ (iii) in Theorem \ref{T:Main}, we need to understand better when positive $*$-linear matrix maps are also completely positive. The title of the paper  \cite{KMcCSZ19}, namely ``There are many more positive maps than completely positive maps,'' may not seem promising, however, as observed in \cite{tHvdM21a,tHvdM22b} under certain structural conditions on the matricization $L$, this does occur. The results presented here come mainly from \cite{tHvdM22b}.

Let $\cL$ be a $*$-linear matrix map as in \eqref{cL} with matricization $L$, Choi matrix $\BL$ and minimal Hill representation as in \eqref{HillRep}. While complete positivity just corresponds to $\BL$ being positive semidefinite, or, equivalently, the Hill matrix $\BH$ positive definite, positivity of $\cL$ is less easy to verify. However, it follows from \cite[Propositions 3.1 and 3.6]{KMcCSZ19} that $\cL$ is positive if and only if
\begin{equation}\label{1PosCon}
(z\otimes x)^* \BL (z\otimes x) \ge 0,\qquad x\in\BF^n,\, z \in \BF^q.
\end{equation}
Via the factorization of the Choi matrix $\BL$ in \eqref{L-BL Hill reps}, it follows that
\begin{align*}
(z\otimes x)^* \BL (z\otimes x) = (z\otimes x)^*\whatC^* \BH^T \whatC (z\otimes x),
\end{align*}
for $x\in\BF^n$ and $z \in \BF^q$. Since $\BH$ is positive definite if and only if $\BH^T$ is positive definite, to see if a positive map $\cL$ is also completely positive comes down to the question whether $y^* \BH^T y\geq 0$ for all $y$ from
\[
\fY_{\whatC}:=\{\whatC (z\otimes x) \colon x\in\BF^n,\, z \in \BF^q\}
\]
is enough to conclude that $\BH^T$ is positive definite. Clearly this is the case when $\fY_{\whatC}=\BF^m$. However, there are also cases of completely positive $\cL$ with $\fY_{\whatC}\neq\BF^m$. Again, we point out that the close relation between minimal Hill representations of $\cL$ makes the choice of the minimal Hill representation irrelevant for the outcome here.

A specific case that is considered in \cite{tHvdM22b} is when there exists a $x\in\BF^n$ so that $\whatC (I_q\otimes x)\in\BF^{m\times q}$ has full row rank or a $z\in\BF^q$ so that $\whatC (z\otimes I_n)\in\BF^{m\times n}$ has full row rank. Since
\[
\whatC (I_q\otimes x)z=
\whatC (z\otimes x)= \whatC (z\otimes I_n)x,
\]
in these two cases clearly $\fY_{\whatC}=\BF^m$ holds, and thus positive maps with this property are also completely positive.

We now recall \cite[Theorem 4.2]{tHvdM22b} which characterize one of the two cases. There is an analogous characterization for the other case, but we do not need it here, hence will not repeat it.

\begin{theorem} \label{property C1}
Let $\cL$ in \eqref{cL} be a $*$-linear map with matricization $L$ and Choi matrix $\BL$. Set $m=\rank \BL$, decompose $L$ as in \eqref{Ldec} and define $\cW$ as in \eqref{cW}. Then for any minimal Hill representation \eqref{HillRep}, for $\whatC$ defined as in \eqref{whatC} there exists a vector $z\in \BF^q$ such that $\whatC(z \otimes I_n)$ has full row-rank if and only if the subspace $\cW$ has the following property:
\begin{align*}
\textup{(C1)} \, &\text{For any linearly independent } X_1, \ldots, X_k\\ &\text{ in } \cW,
\text{ there exists a } v \in \BF^q \text{ such that } \\ &X_1v, \ldots, X_kv
 \text{ is linearly independent in } \BF^n.
\end{align*}
Hence, if \textup{(C1)} holds, then $\fY_{\widehat{C}}=\BF^m$  and positivity and complete positivity of $\cL$ coincide. Clearly, this can only happen if $\dim \cW\leq n$.
\end{theorem}

As a first example, if the subspace $\cW$ consists of only upper triangular Toeplitz matrices or of only lower triangular Toeplitz matrices of size $n\times n$, then it is easy to see that (C1) holds; one simply takes the first standard unit vector $e_1$ or the last $e_n$, respectively.

More examples can be constructed by the following rules. If a subsepace $\cW\subset \BF^{n\times q}$ satisfies (C1), then:
\begin{itemize}
\item[(i)] $\ov{\cW}$ also satisfies (C1);

\item[(ii)] $P \cW Q$ also satisfies (C1) for any invertible $P \in \BF^{n\times n}$ and $Q \in \BF^{q\times q}$;

\item[(iii)] any subspace $\cV$ of $\cW$ satisfies (C1);

\item[(iv)] if $\cZ$ is a subspace of $\BF^{n' \times q'}$ which satisfies (C1), then $\cW\oplus \cZ\subset\BF^{(n+ n') \times (q+q')}$ also satisfies (C1).

\end{itemize}
Only rule (iv) is not trivial to prove \cite[Lemma 4.8]{tHvdM22b}. Using the above rules, starting with Jordan blocks, one can put together $\{D\}_\BF''$ for any matrix $D\in \BF^{n \times n}$. Hence, one of the main results of \cite[Theorem 4.13]{tHvdM22b} emerges.

\begin{theorem}
For any $D\in\BF^{n \times n}$, the algebra $\{D\}_\BF''$ satisfies (C1).
\end{theorem}

From here the proof of the implication (ii) $\Rightarrow$ (iii) in Theorem \ref{T:Main} easily follows, since we know that the matricization $L_{A,B}$ of $\cL_{A,B}$ is contained in $\ov{\{A^*\}_\BF''}\otimes \{A^*\}_\BF''$. While this does not imply that the subspace $\cW$ associated with $\cL_{A,B}$ is equal to the algebra $\{A^*\}_\BF''$, it is contained in $\{A^*\}_\BF''$, since all block entries of $L_{A,B}$ are in $\{A^*\}_\BF''$, so that by rule (iii) the subspace $\cW$ of $\cL_{A,B}$ still has property (C1).

\section{The proof of the main theorem}\label{S:Proof}

In this section we complete the proof of Theorem \ref{T:Main} by proving the remaining implication (iv) $\Rightarrow$ (i). For this purpose we first prove a few auxiliary results that can be seen as steps in the proof. Of these, only one relies on the suboptimality assumption. In the first result we use a minimal Hill representation of the linear matrix map $\cL_{A,B}$ to derive an identity that is the basis of the remainder of the proof.

\begin{lemma}\label{L:SkewInter}
Let $A,B\in\BF^{n \times n}$ with $A$ Lyapunov regular and let \eqref{cLAB-Hill1} be a minimal Hill representation of $\cL_{A,B}$. Then, with $\bC$ as in \eqref{cLAB-Hill1}, for all $X\in\BF^{n\times n}$ we have
\begin{align} \label{LyapID}
\begin{aligned}
&B^*X-\bC^*\left(\BH_{A,B}\otimes A^*X\right)\bC=\\
&\qquad\qquad=\bC^*\left(\BH_{A,B}\otimes XA\right)\bC-XB.
\end{aligned}
\end{align}
Assume, in addition, $B \in \mathcal{C}_{\cL}(A)\bigcap \{A\}_{\BF}''$, so that $\cL_{A,B}$ is completely positive and thus $\BH_{A,B} \in \cP_m$. Factor $\BH_{A,B}=P^*P$ for some invertible $P \in \BF^{m\times m}$ and for $R \in \BF^{n \times n}$ set
\begin{align}\label{LRMR}\begin{aligned}
L_R&:=\begin{bmatrix}
R \\ \left(P \otimes R\right)\bC
\end{bmatrix}, \quad \text{and} \\
M_R&:=\begin{bmatrix}
RB \\ -\left(P \otimes RA\right)\bC
\end{bmatrix}.\end{aligned}
\end{align}
Then for all $R,R' \in \BF^{n \times n}$ we have
\begin{equation}\label{skewinter}
M_{R'}^*L_R=-L_{R'}^*M_R.
\end{equation}
\end{lemma}

\begin{proof}[\bf Proof]
From $\cL_{B}=\cL_B\cL_{A}^{-1}\cL_A$ and since the Kronecker product is left-distributive, it follows for each $X \in \BF^{n \times n}$ that
\begin{align*}
XB+B^*X
&=\cL_{B}(X)=\cL_B\cL_{A}^{-1}\cL_{A}(X)\\&=\cL_{A,B}(\cL_A(X))\\
&=\mathbf{C}^*\left(\BH_{A,B}\otimes \cL_A(X)\right)\mathbf{C}\\
&= \mathbf{C}^*\left(\BH_{A,B}\otimes \left( XA+A^*X \right) \right)\mathbf{C}\\
&= \mathbf{C}^*\left(\BH_{A,B}\otimes XA\right)\mathbf{C}\\ &\qquad \qquad +\mathbf{C}^*\left(\BH_{A,B}\otimes A^*X\right)\mathbf{C}.
\end{align*}
From this identity, \eqref{LyapID} follows. In case $B\in \mathcal{C}_{\cL}(A)\bigcap \{A\}_{\BF}''$, then \cite{tHvdM22b} implies that $\BH_{A,B}$ is positive definite. Factoring $\BH_{A,B}=P^*P$ as stated, and taking $X=R'^*R$ for $R,R' \in \BF^{n \times n}$ in \eqref{LyapID}, we obtain \eqref{skewinter}.
\end{proof}

To proceed we require a variation on the Douglas Factorization Lemma which we state below as Proposition \ref{P:DouglasSkew}. First we introduce some further notation and derive a useful lemma. For any finite collection $\bR=\{R_1,\ldots,R_k\}\subset \BF^{n \times n}$ set
\begin{align}\label{LbEMbR}\begin{aligned}
&L_\bR=\mat{L_{R_1}&\cdots L_{R_k}}\quad \text{and} \\  &M_\bR=\mat{M_{R_1}&\cdots M_{R_k}}.
\end{aligned}\end{align}
Since $M_{R'}^*L_R=-L_{R'}^*M_R$ for all $R,R'\in\BR^{n \times n}$, we also have
\[
M_{\bR}^*L_{\bR}=-L_{\bR}^*M_{\bR}
\]
for each finite collection $\bR\subset \BF^{n \times n}$. Further, define the subspaces
\begin{equation}\label{cVcW}
\begin{aligned}
\cU=\textup{span}\left\{L_Rx: \, R \in \BF^{n \times n}, \, x \in \BF^n\right\},\\
\cV=\textup{span}\left\{M_Rx: \, R \in \BF^{n \times n}, \, x \in \BF^n\right\}
\end{aligned}
\end{equation}
and note that $\tu{Ran}\, L_\bR\subset\cU$ and $\ran M_\bR\subset\cV$ for each finite collection $\bR\subset \BF^{n \times n}$.

Next we prove a lemma that will be of use in the proof of Proposition \ref{P:DouglasSkew} below.

\begin{lemma}\label{L:cVcWform}
Let $A,B\in\BF^{n \times n}$ with $A$ Lyapunov regular and  $B \in \mathcal{C}_{\cL}(A)\bigcap \{A\}_{\BF}''$.  Then the subspaces $\cU$ and $\cV$ defined by \eqref{cVcW} are of the form $\cU=\wtil{\cU}\otimes \BF^n$ and $\cV=\wtil{\cV}\otimes \BF^n$ with
\begin{equation}\label{tilcVtilcW}
\begin{aligned}
\wtil{\cU}= & \left\{x \in \BF^{m+1}: \exists\, 0\neq y \in \BF^n \right. \\ &\left. \qquad \qquad \quad \quad\text{ so that }x \otimes y \in \cU \right \}, \\
\widetilde{\cV}= &\left\{x \in \BF^{m+1}: \exists\, 0\neq y \in \BF^n \right. \\ &\left. \quad \quad\qquad \qquad \text{ so that }x \otimes y \in \cV \right \}.
\end{aligned}
\end{equation}
\end{lemma}

\begin{proof}[\bf Proof]
We prove the formula for $\cU$. The proof for $\cV$ goes analogously.

Note that $\left(I_{m+1} \otimes R\right)\cU\subset \cU$ for any matrix $R \in \BF^{n \times n}$, which follows from the fact that
\begin{equation}\label{Rcalc}
\left(I_{m+1} \otimes R'\right)L_R=L_{R'R},
\end{equation} for all $R,R' \in \BF^{n \times n}$.
The identity $\BF^{(m+1)\cdot n}=\BF^{m+1}\otimes \BF^n$ shows that vectors in $\BF^{(m+1)\cdot n}$ are of the form
\[
\sum_{i=1}^p x_i \otimes y_i, \quad \text{where} \quad x_i \in \BF^{m+1} \quad \text{and} \quad y_i \in \BF^{n}.
\]
Take any element $v=\sum_{i=1}^p x_i\otimes y_i \in \cU$.  By the same argument as in \cite[Theorem 6.1]{GGK03}, we can ensure that $y_1, \ldots, y_p$ are linearly independent. In particular, we then have $p\leq n$. From
\[
\left(I_{m+1} \otimes R\right)\sum_{i=1}^p x_i \otimes y_i = \sum_{i=1}^p x_i\otimes Ry_i \in \cU,
\]
and by choosing $R$ so that $Ry_i=0$ for $i\ne j$, for some $j$, and $Ry_j=z$ with $z\in\BF^n$ arbitrary, which is possible since $p\leq n$, we see that \begin{align*} \label{subspace v} x_j \otimes \BF^n \subseteq \cU \quad \text{for all} \quad j.\end{align*} Using this property we can see at once that $\widetilde{\cU}$ is a subspace. Moreover, it follows that $x_j\in\wtil{\cU}$ for each $j$. However, this in turn implies that $v\in \widetilde{\cU}\otimes \BF^n$. Hence we proved that $\cU \subseteq \widetilde{\cU}\otimes \BF^n$.

For the reverse inclusion $\widetilde{\cU}\otimes \BF^n \subseteq \cU$, take $\sum_{i=1}^p x_i \otimes y_i \in \widetilde{\cU} \otimes \BF^n.$ Hence $x_i \in \widetilde{\cU}$ for all $i$, which means there exist non-zero vectors $z_i$ such that $x_i\otimes z_i \in \cU$. Again we make use of $\left(I_{m+1}\otimes R\right)\cU\subset \cU$ for all $R \in \BF^{n \times n}$. Vary $R$ so that $y_i=Rz_i$ for all $i$. Then it follows that $x_i\otimes y_i \in \cU$ for all $i$ and therefore also $\sum_{i=1}^p x_i \otimes y_i \in \cU$. Thus the inclusion $\widetilde{\cU}\otimes \BF^n \subseteq \cU$ holds and we have proved that $\cU=\widetilde{\cU}\otimes \BF^n.$
\end{proof}

Now we are ready for our variation on the Douglas Factorization Lemma.

\begin{proposition}\label{P:DouglasSkew}
Let $A,B\in\BF^{n \times n}$ with $A$ Lyapunov regular and  $B \in \mathcal{C}_{\cL}(A)\bigcap \{A\}_{\BF}''$. For $R\in\BF^{n \times n}$, define $L_R$ and $M_R$ as in \eqref{LRMR}. Then there exists a matrix $\whatS\in\BF^{(1+m)n \times (1+m)n}$ so that
\begin{equation}\label{SLR=MR}
\whatS L_R=M_R\quad\mbox{for each } R\in\BF^{n \times n}
\end{equation}
if and only if $\kr L_{\bR} \subset \kr M_{\bR}$ for all finite collections $\bR\subset \BF^{n \times n}$. Moreover, it suffices to verify $\kr L_{\bR} \subset \kr M_{\bR}$ for a single collection $\bR\subset \BF^{n \times n}$ with the property that $\ran L_{\bR}=\cU$, with $\cU$ as in \eqref{cVcW}. Furthermore, if there exists a matrix $\whatS\in\BF^{(1+m)n \times (1+m)n}$ satisfying \eqref{SLR=MR}, then without loss of generality one may assume $\whatS=S\otimes I_n$ for a $S\in\BF^{(1+m) \times (1+m)}$ that satisfies $S^*=-S$.
\end{proposition}

\begin{proof}[\bf Proof]
We will divide the proof into three steps.

\noindent{\bf Step 1.} First we consider what Douglas Factorization Lemma tells us for fixed finite collections of $\BF^{n \times n}$ and prove the necessity of the existence criteria.

Fix a finite collection $\bR\subset\BF^{n \times n}$. By Douglas Factorization Lemma there exists a $\whatS_\bR\in\BF^{(1+m)n \times (1+m)n}$ so that $\whatS_\bR L_\bR=M_\bR$ if and only if $\kr L_{\bR} \subset \kr M_{\bR}$. Moreover, set
\begin{equation}\label{cVRcWR}
\cU_\bR=\ran L_\bR \ands \cV_\bR=\ran M_\bR.
\end{equation}
Then $\whatS_\bR$ satisfying $\whatS_\bR L_\bR=M_\bR$ is uniquely determined if one further demands that $\cU_\bR^{\perp} \subset \kr \whatS_\bR$, in which case $\ran \whatS_\bR=\cV_\bR$. In case $\bR'\subset\BF^{n \times n}$ is another finite collection so that $\bR\subset \bR'$, then $\cU_\bR\subset \cU_{\bR'}$, $\cV_\bR\subset \cV_{\bR'}$ and we have that $L_\bR$ and $M_\bR$ correspond to restrictions of $L_{\bR'}$ and $M_{\bR'}$ to the appropriate block columns (ignoring possible reordering of the entries of $\bR$ and $\bR'$), and consequently we have that $\whatS_{\bR'}|_{\cU_\bR}=\whatS_\bR$.

Clearly, if $\whatS L_R=M_R$ for each $R\in\BF^{n \times n}$, then also $\whatS L_\bR=M_\bR$ for all finite collections $\bR\subset\BF^{n \times n}$, so that the necessity of $\kr L_{\bR} \subset \kr M_{\bR}$ follows by the above observations.\smallskip

\noindent{\bf Step 2.} Next we prove the sufficiency of the existence criteria, show that we can work with a single finite collection $\bR$ in $\BF^{n \times n}$ so that $\ran L_\bR=\cU$ and that we can make $\whatS$ unique by demanding $\whatS|_{\cU^\perp}=0$. To this end, assume that $\kr L_{\bR} \subset \kr M_{\bR}$ holds for all finite collections $\bR\subset\BF^{n \times n}$, so that there exists a $\whatS_\bR\in\BF^{(1+m)n \times (1+m)n}$ so that $\whatS_\bR L_\bR=M_\bR$ for all finite collections $\bR\subset\BF^{n \times n}$. Take $\whatS_\bR$ with $\whatS_\bR|_{\cU_\bR^\perp}=0$, so that $\whatS_\bR$ is uniquely determined. Now take a finite collection $\bR\subset\BF^{n \times n}$ so that $\cU_\bR=\cU$, which exists by a dimension argument. Let $R\in \BF^{n \times n}$ be arbitrary and set $\bR'=\bR\cup\{R\}$. Then $\cU=\cU_\bR\subset \cU_{\bR'}\subset\cU$, so that $\cU_\bR= \cU_{\bR'}=\cU$ and we have $\whatS_\bR=\whatS_{\bR'}$ because they coincide on both $\cU$ and $\cU^\perp$. Comparing the block column corresponding to the position of $R$ in the identity $\whatS_\bR L_{\bR'}=M_{\bR'}$, it follows that $\whatS_\bR L_{R}=M_{R}$ holds for arbitrarily chosen $R\in\BF^{n \times n}$, with $\whatS_\bR$ no depending on $R$. With the observations of the first part we note that $\whatS_\bR$ exists if and only if $\kr L_{\bR} \subset \kr M_{\bR}$, which we only have to verify for the selected finite collection $\bR$. It also follows that $\whatS=\whatS_\bR$ is uniquely determined if we demand in addition that $\whatS|_{\cU^\perp}=0$.\smallskip

\noindent{\bf Step 3.} In the third step we prove the final claim, that is, that without loss of generality $\whatS$ has the form $\whatS=S\otimes I_n$ with  $S\in\BF^{(1+m) \times (1+m)}$ satisfying $S^*=-S$. Let $\bR\subset\BF^{n \times n}$ be a finite collection so that $\cU_\bR=\cU$. Assume that $\kr L_{\bR} \subset \kr M_{\bR}$, so that $\whatS$ be as in Step 2 exists, in particular with $\whatS|_{\cU^\perp}=0$. Then we can decompose $\whatS$ as
\[
\widehat{S}=\begin{bmatrix}
    \widehat{S}_{11} & 0 \\ \widehat{S}_{21} & 0
   \end{bmatrix}: \begin{bmatrix}
    \cU \\ \cU^\perp
   \end{bmatrix} \to \begin{bmatrix}
    \cU \\ \cU^\perp
   \end{bmatrix}.
\]
Recall that Lemma \ref{L:SkewInter} implies that $L_\bR^*M_\bR=-M_\bR^*L_\bR$. Since $\whatS L_\bR =M_\bR$, it follows that
\begin{align*}
0 &=L_\bR^*M_\bR+M_\bR^*L_\bR=L_\bR^*\whatS L_\bR+L_\bR^* \whatS^* L_\bR\\
&=L_\bR^*(\whatS+ \whatS^*)L_\bR.
\end{align*}
The fact that $\ran L_\bR=\cU$ then implies that  $S_{11}+S_{11}^*=0$, hence that $S_{11}$ is skew-Hermitian.

Next we show that
\begin{equation} \label{on image}
\widehat{S}\left(I_{m+1}\otimes R\right)v=\left(I_{m+1}\otimes R\right)\widehat{S}v, 
\end{equation} for $\  \mbox{$R \in \BF^{n \times n}$ and $v\in\cU$}.$
To see that this is the case, in \eqref{SLR=MR} replace $R$ by $RR'$ which gives
    \[\widehat{S}L_{RR'}
    x=M_{RR'}x, \quad \text{for all} \quad x \in \BF^n.\]
Then, using \eqref{Rcalc} and a similar identity with $L_R$ replaced by $M_R$, we see that
    \begin{align*}
&\widehat{S}\left(I_{m+1}\otimes R\right)L_{R'}x
=\widehat{S}L_{RR'}x
=M_{RR'}x\\
&\qquad=\left(I_{m+1}\otimes R\right) M_{R'}x= \left(I_{m+1}\otimes R\right)\widehat{S} L_{R'}x. \end{align*}
By varying $R'$ from the entries in $\bR$ and $x\in\BF^n$ and taking linear combinations, it follows that \eqref{on image} indeed holds for all $v \in \cU.$

From the fact that $\cU=\widetilde{\cU} \otimes \BF^n$ for the subspace $\widetilde{\cU}\subseteq \BF^{m+1}$ as defined in \eqref{tilcVtilcW}, which gives $\cU^\perp=\widetilde{\cU}^\perp \otimes \BF^n$, it follows that \[I_{m+1}\otimes R = \begin{bmatrix}
    I_{\widetilde{\cU}}\otimes R & 0 \\ 0 & I_{\widetilde{\cU}^\perp} \otimes R
   \end{bmatrix}:\begin{bmatrix}
    \cU \\ \cU^\perp
   \end{bmatrix} \to \begin{bmatrix}
    \cU \\ \cU^\perp
   \end{bmatrix}.\]
Then \eqref{on image} yields
\begin{align*}
\begin{bmatrix}
\widehat{S}_{11}\left(I_{\widetilde{\cU}}\otimes R\right) \\ \widehat{S}_{21}\left(I_{\widetilde{\cU}}\otimes R\right)\end{bmatrix}
&=\widehat{S}\left(I_{m+1}\otimes R\right)|_{\cU}\\&
=\left(I_{m+1}\otimes R\right)\widehat{S}|_{\cU}\\
&=\begin{bmatrix} I_{\widetilde{\cU}}\otimes R & 0 \\ 0 & I_{\widetilde{\cU}^\perp} \otimes R
   \end{bmatrix}\begin{bmatrix}
    \widehat{S}_{11}\\\widehat{S}_{21}
   \end{bmatrix}\\&=\begin{bmatrix}
    \left(I_{\widetilde{\cU}}\otimes R\right)\widehat{S}_{11} \\ \left(I_{\widetilde{\cU}^\perp} \otimes R\right)\widehat{S}_{21}
   \end{bmatrix}.
   \end{align*}
   From $ \widehat{S}_{11}\left(I_{\widetilde{\cU}}\otimes R\right)=\left(I_{\widetilde{\cU}}\otimes R\right)\widehat{S}_{11}$ and $\widehat{S}_{21}\left(I_{\widetilde{\cU}}\otimes R\right)=\left(I_{\widetilde{\cU}^\perp} \otimes R\right)\widehat{S}_{21}$, we see that the $n \times n$ blocks in $\widehat{S}_{11}$ and $\widehat{S}_{21}$ commute with all matrices $R \in \BF^{n \times n}$. Thus all $n \times n$ blocks in $\widehat{S}_{11}$ and $\widehat{S}_{21}$ have the form of a scalar times $I_n$ and therefore we can write \[\widehat{S}_{11}=S_{11}\otimes I_n \quad \text{and} \quad \widehat{S}_{21}=S_{21}\otimes I_n.\]
Since $\whatS_{11}$ is skew-Hermitian, clearly also $S_{11}$ is skew-Hermitian.

Now simply replace the 0 in the $(1,2)$-entry in the $2\times 2$ block decomposition $\whatS$ by $-\whatS_{21}^*$, that is, we chance $\whatS$ to
\[\whatS=\mat{\whatS_{11}&-\whatS_{21}^*\\\whatS_{21}&0}=S\otimes I_n\] with \[S=\mat{S_{11}&-S_{21}^*\\S_{21}&0}.\]
It is clear that the adjusted $\whatS$ has the required form. Since  making chances in $\whatS|_{\cU^\perp}$ does not alter the validity of \eqref{SLR=MR}, this identity remains true for the new $\whatS$.
\end{proof}

\begin{remark}
The matrix $\whatS$ satisfying \eqref{SLR=MR} with the additional properties in the last claim of Proposition \ref{P:DouglasSkew} is not unique. However, the proof shows precisely the level of freedom that is left. Indeed, the identity \eqref{skewinter} uniquely determines $\whatS|_{\cU}$. Moreover, the property that $\whatS^*=-\whatS$ determined the right upper corner in the $2 \times 2$ decomposition of $\whatS$. However, for the right lower corner, instead of 0, we are free to choose any skew-Hermitian linear map $\whatS_{22}$ from $\cU^\perp$ into itself, of the form $\whatS_{22}=S_{22}\otimes I_n$. Hence, the level of freedom is restricted to skew-Hermitian linear maps on $\wtil{\cU}^\perp$. Varying the skew-Hermitian linear map $S_{22}$, will likely lead to different solutions via the formula in Lemma \ref{L:ToReal} below, but probably not all.
\end{remark}

Next we provide a condition under which the matrix $\whatS$ satisfying \eqref{SLR=MR} exists, which includes the suboptimal case.

\begin{lemma}\label{L:subopt-suf}
Let $A,B\in\BF^{n \times n}$ with $A$ Lyapunov regular and $B \in \mathcal{C}_{\cL}(A)\bigcap \{A\}_{\BF}''$. Assume that the matrices $C_1,\ldots,C_m\in \{A^*\}_\BF''$ in the minimal Hill representation \eqref{cLAB-Hill1} of $\cL_{A,B}$ are such that  $\tu{span}\{C_1,\ldots,C_m,I_n\}=\{A^*\}_\BF''$. Then $\kr L_{\bR} \subset \kr M_{\bR}$ for all finite collections $\bR\subset \BF^{n \times n}$. In particular, this is the case when $m=m_\tu{max}$.
\end{lemma}

\begin{proof}[\bf Proof]
By assumption, we have \\ $\tu{span}\{C_1^*,\ldots,C_m^*,I_n\}=\{A\}_\BF''$. Since $\{A\}_\BF''$ is a unital algebra that contains both $A,B$ and $C_1^*,\ldots,C_m^*$, it follows that $AC_1^*,\ldots,AC_m^*,B\in \tu{span}\{C_1^*,\cdots,C_m^*,I_n\}$.

Consider an arbitrary finite collection $\bR=\{R_1,\ldots,R_k\}\subset\BF^{n \times n}$. Assume $x\in \kr L_{\bR}$. Then $x^T=(x_1^T,\ldots,x_k^T)$ for vectors $x_1,\ldots,x_k\in \BF^n$ and we have
\[
\sum_{j=1}^k R_j x_j=0\] and \[ \sum_{j=1}^k R_j C^*_i  x_j=0\mbox{ for $i=1,\ldots,m$}.
\]
Since $AC^*_1 ,\ldots,AC^*_m,B\in \tu{span}\{C^*_1,\cdots,C^*_m,I_n\}$, we can write
\[
B= \upsilon I_n + \sum_{i=1}^m \eta_i C^*_i \] and \[AC^*_l=\sigma_l I_n + \sum_{i=1}^m \zeta_{il} C^*_i,
\]
for $l=1,\ldots,m$. Using these formulas it follows that
\begin{align*}
\sum_{j=1}^k R_j B x_j
&= \sum_{j=1}^k R_j \left( \upsilon x_j+ \sum_{i=1}^m \eta_i C^*_i x_j \right)\\
&= \upsilon \sum_{j=1}^k R_j x_j + \sum_{i=1}^m \eta_i \sum_{j=1}^k R_j C^*_i x_j \\&= 0
\end{align*}
and, similarly, for $l=1,\ldots,m$ we have
\begin{align*}
\sum_{j=1}^k R_j A C^*_l x_j
&= \sum_{j=1}^k R_j \left(\sigma_l x_j + \sum_{i=1}^m \zeta_{il} C^*_i x_j \right)\\
&= \sigma_l\sum_{j=1}^k R_j x_j +  \sum_{i=1}^m \zeta_{il} \sum_{j=1}^k R_j C^*_i x_j \\&= 0.
\end{align*}
From these identities if follows that $M_{\bR}x=0$. Hence $\kr L_{\bR}\subset \kr M_{\bR}$, as claimed.

In case $m=m_\tu{max}=\dim \{A\}_\BF''$, since $C^*_1,\ldots,C^*_m$ are linearly independent, it follows that $\tu{span}\{C^*_1,\ldots,C^*_m\}=\{A\}_\BF''$, and since the algebra $\{A\}_\BF''$ is unital, nothing is added by adding $I_n$ to the span.
\end{proof}

Finally, we show how existence of the matrix $\whatS$ satisfying \eqref{SLR=MR} leads to a solution to our interpolation problem. It is only here that we have to restrict to the case where $\BF=\BR$. The result uses state space realization theory for functions in $\PRO$, which is well known, cf., \cite{ZDG96}.

\begin{lemma}\label{L:ToReal}
Let $A,B\in\BR^{n \times n}$ with $A$ Lyapunov regular and  $B \in \mathcal{C}_{\cL}(A)\bigcap \{A\}_{\BR}''$. For $R\in\BF^{n \times n}$, define $L_R$ and $M_R$ as in \eqref{LRMR}. Assume that there exists a matrix $S\in\BR^{(1+m)\times(1+m)}$ so that
\begin{align}\begin{aligned}\label{Sprops}
(S\otimes I_n) L_R&=M_R, \ \  R\in\BR^{n\times n},\\ \mbox{ and} \quad S^T&=-S.
\end{aligned}\end{align}
Decompose $S$ as
\[
S=\mat{0&\ell\\-\ell^T & -M},
\] for $M=-M^T\in \BR^{m\times m}$ and $\ell\in\BR^{1 \times m}$. Set
\begin{equation}\label{fform}
f(z)=\ell(zI_m-M)^{-1}\ell^T.
\end{equation}
Then $f\in\PRO$ and $f(A)=B$.
\end{lemma}

\begin{proof}[\bf Proof]
Let $S\in\BR^{(1+m)\times(1+m)}$ be so that \eqref{Sprops} holds. Decompose $S$ as
\[
S=\mat{0&\ell\\-\ell^T & -M},
\] for $M=-M^T\in \BR^{m\times m}$ and $\ell\in\BR^{1 \times m}.$
Since \eqref{Sprops} holds for all $R \in \BR^{n \times n}$, it follows that
\begin{align*} 
&\left( \begin{bmatrix} 0 & \ell \\ -\ell^T & -M \end{bmatrix} \otimes I_n\right)\begin{bmatrix}
    I_n \\ \left(P \otimes I_n\right)\mathbf{C}
    \end{bmatrix}=\\
&\qquad=  \widehat{S}\begin{bmatrix}
    I_n \\ \left(P \otimes I_n\right)\mathbf{C}
    \end{bmatrix}=\begin{bmatrix}
    B \\ -\left(P \otimes A\right)\mathbf{C} \\
    \end{bmatrix}.
    \end{align*}
From this identity we get
\begin{align} \label{eq1} \left(\ell \otimes I_n\right)\left(P \otimes I_n\right)\mathbf{C}=B
\end{align}
and
\begin{align*}
&-\left(\ell^T \otimes I_n\right) - \left(M \otimes I_n\right)\left(P\otimes I_n\right)\mathbf{C}\\&\qquad \qquad=-\left(P\otimes A\right)\mathbf{C}\\
&\qquad\qquad=-\left(I_m\otimes A\right)\left(P\otimes I_n\right)\mathbf{C}
\end{align*}
so that
\begin{align*}\left(\ell^T \otimes I_n\right) =\left(I_{m}\otimes A-M\otimes I_n\right)\left(P\otimes I_n\right)\mathbf{C} \end{align*} and thus \begin{align} \label{vgl met A en B} \begin{aligned} &\left(I_{m}\otimes A-M\otimes I_n\right)^{-1}\left(\ell^T \otimes I_n\right) \\&\qquad \qquad \qquad \qquad \qquad \quad=\left(P\otimes I_n\right)\mathbf{C}.\end{aligned} \end{align} To see that $I_{m}\otimes A-M\otimes I_n$ is a non-singular matrix, recall that $A$ is Lyapunov regular, which implies that $A$ has no eigenvalues on $\imag\BR$, and since $M$ is a skew-symmetric matrix its entire spectrum lies on $\imag\BR$. By Theorem 4.4.5 in \cite{HJ91} we know that $I_{m}\otimes A+M^T\otimes I_n$ has its eigenvalues as all possible pairwise sums of the eigenvalues of $A$ and $M^T,$ where $M^T=-M$ since it is also skew-symmetric. Hence $0$ can never be an eigenvalue of $I_{m}\otimes A-M\otimes I_n$. Plugging \eqref{vgl met A en B} into \eqref{eq1} gives \[B=\left(\ell \otimes I_n\right)\left(I_{m}\otimes A-M\otimes I_n\right)^{-1}\left(\ell^T \otimes I_n\right).\] Now define $f$ to be the rational function given by \[f(z)=\ell\left(zI_m-M\right)^{-1}\ell^T.\] Then \begin{align*} f(A)&= \left(\ell \otimes I_n\right)\left(I_{m}\otimes A-M\otimes I_n\right)^{-1}\left(\ell^T \otimes I_n\right)\\&=B.
\end{align*}
It follows from basic realization theory that $f$ given by \eqref{fform} with $M$ skew-Hermitian is in $\PRO$, cf., \cite[Section 2]{tHN21} and references given there.
\end{proof}

{\bf Proof of Theorem \ref{T:Main}.} The proof of Theorem \ref{T:Main} now follows by simply putting the pieces together. We had already argued that it remains to prove the implication (iv) $\Rightarrow$ (i). So take a minimal Hill representation of $\cL_{A,B}$ and assume the Hill-Pick matrix $\BH_{A,B}$ is positive definite. In particular, $m=m_\tu{max}$. Then Lemma \ref{L:subopt-suf} shows that the conditions required in Proposition \ref{P:DouglasSkew} are satisfied, so that we obtain a matrix $S$ satisfying \eqref{SLR=MR}. From this in turn a function $f$ in $\PRO$ satisfying $f(A)=B$ can obtained via Lemma \ref{L:ToReal}, establishing item (i) of Theorem \ref{T:Main}, as desired. \qed

\section{Conclusion}

In this paper we prove a Nevanlinna-Pick type interpolation theorem in the class of positive real odd functions ($\PRO$) where the interpolation point is a real matrix, partially resolving a conjecture of Cohen and Lewkowicz from \cite{CL09}. The result is proved under an additional condition referred to as ``suboptimal,'' since it corresponds to the suboptimal case in the classical Nevanlinna-Pick setting. We do not only obtain a necessary and sufficient condition for existence of the solution, but also provide an explicit way to construct a solution.

\section*{Acknowledgments}
This work is based on research supported in part by the National Research Foundation of South Africa (NRF) and the DSI-NRF Centre of Excellence in Mathematical and Statistical Sciences (CoE-MaSS). Any opinion, finding and conclusion or recommendation expressed in this material is that of the authors and the NRF and CoE-MaSS do not accept any liability in this regard.


\end{multicols}

\end{document}